\documentclass[12pt,a4paper,reqno]{amsart}

\usepackage{amsmath,amsfonts,amsthm,amssymb,graphicx}
\usepackage[shortlabels]{enumitem}

\usepackage[dvipsnames]{xcolor}
\usepackage{hyperref}

\usepackage[marginratio=1:1,totalwidth=15.75cm,totalheight=22.275cm]{geometry}

\usepackage{subcaption}

\usepackage{graphicx}

\theoremstyle{plain}
\newtheorem{thm}{Theorem}[section]
\newtheorem{lem}[thm]{Lemma}
\newtheorem{prop}[thm]{Proposition}
\newtheorem{cor}[thm]{Corollary}

\renewcommand{\subset}{\subseteq}
\renewcommand{\theta}{\vartheta}
\renewcommand{\phi}{\varphi}

\theoremstyle{definition}

\theoremstyle{remark}
\newtheorem{rmk}[thm]{Remark}

\newcommand{\eps}{\varepsilon}
\newcommand{\interior}{\operatorname{int}}

\newcommand{\DD}{\mathbb{D}}
\newcommand{\dist}{\operatorname{dist}}

\numberwithin{equation}{section}
\newcommand \C{\mathbb{C}}
\newcommand \Ch{\widehat{\mathbb{C}}}
\newcommand \N{\mathbb{N}}

\newcommand \D{\mathbb{D}}

\newcommand*{\defeq}{\mathrel{\vcenter{\baselineskip0.5ex \lineskiplimit0pt
                     \hbox{\scriptsize.}\hbox{\scriptsize.}}}%
                     =}

\renewcommand{\geqslant}{\geq}
\newcommand{\Fill}{\operatorname{fill}}

\makeatletter
\@namedef{subjclassname@2020}{%
  \textup{2020} Mathematics Subject Classification}
\makeatother

\begin{document}

\title{Bounded Fatou and Julia components of meromorphic functions}
\author[D. Mart\'i-Pete \and L. Rempe \and J. Waterman]{David Mart\'i-Pete \and Lasse Rempe \and James Waterman}

\address{Department of Mathematical Sciences\\ University of Liverpool\\ Liverpool L69 7ZL\\ United Kingdom\textsc{\newline \indent \href{https://orcid.org/0000-0002-0541-8364}{\includegraphics[width=1em,height=1em]{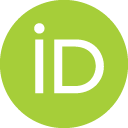} {\normalfont https://orcid.org/0000-0002-0541-8364}}}} 
\email{david.marti-pete@liverpool.ac.uk}

\address{Department of Mathematical Sciences\\ University of Liverpool\\ Liverpool L69 7ZL\\ United Kingdom\textsc{\newline \indent \href{https://orcid.org/0000-0001-8032-8580}{\includegraphics[width=1em,height=1em]{orcid2.png} {\normalfont https://orcid.org/0000-0001-8032-8580}}}} 
\email{lrempe@liverpool.ac.uk}

\address{Institute for Mathematical Sciences\\ Stony Brook University\\ Stony Brook NY 11794--3660\\ USA\textsc{\newline \indent \href{https://orcid.org/0000-0001-7266-0292}{\includegraphics[width=1em,height=1em]{orcid2.png} {\normalfont https://orcid.org/0000-0001-7266-0292}}}}
\email{james.waterman@stonybrook.edu}


\date{\today}


\begin{abstract}
We completely characterise the bounded sets that arise as components of the Fatou and Julia sets of meromorphic functions. On the one hand, we prove that a bounded domain is a Fatou component of some meromorphic function if and only if it is regular. On the other hand, we prove that a planar continuum is a Julia component of some meromorphic function if and only if it has empty interior. We do so by constructing meromorphic functions with wandering continua using approximation theory. 
\end{abstract}

\maketitle

\section{Introduction}

Let $f\colon\C\to \Ch$ be a meromorphic function. 
 The \emph{Fatou set} of $f$, $F(f)$, is the set of points of the Riemann sphere whose orbit under $f$ remains stable
  under small perturbations.
 More formally, $F(f)$ is the largest open set on which the iterates of~$f$ are defined and form a normal family. Its complement, $J(f) \defeq \C\setminus F(f)$, is called the \emph{Julia set} of~$f$. 
 We refer to the connected components of $F(f)$ and $J(f)$ as \textit{Fatou components} and \textit{Julia components}, respectively. Note that throughout this paper when we say component, we mean connected component. For an introduction
 to the iteration of meromorphic functions, see \cite{bergweiler93}.
 
The geometric structure of Julia sets is normally very complicated, and the same
  tends to be true for the boundaries of Fatou components. Therefore it seems
  surprising that we are able to give the following 
  complete characterisation of which bounded sets can arise as 
   Fatou and Julia components of meromorphic functions. (Compare
   Figure~\ref{fig:continuum}.) Recall that a domain 
     $U\subseteq \C$ is called \textit{regular} if $\textup{int}(\overline{U})=U$.
 
 \begin{thm}[Bounded components of the Fatou set]
Let $U\subset\C$ be a bounded domain, that is, $U$ is non-empty, open, connected and bounded. The following are equivalent:
 \begin{enumerate}[(a)]
    \item $U$ is regular;\label{item:Uregular} 
    \item there is a meromorphic function $f$ such that $U$ is a Fatou component of $f$.\label{item:Fatoucomponent}
  \end{enumerate}\label{thm:comp-fatou}
\end{thm}
   
\begin{thm}[Bounded components of the Julia set]
Let $X\subset\C$ be a continuum, that is, $X$ is non-empty, closed, connected and bounded. The following are equivalent:
 \begin{enumerate}[(a)]
    \item $X$ has empty interior; 
    \item there is a meromorphic function $f$ such that $X$ is a Julia component of $f$.
  \end{enumerate}\label{thm:comp-julia}
\end{thm}   

 The domain $U$ in Theorem~\ref{thm:comp-fatou} will be a \emph{wandering domain} of the function $f$;
   that is, a Fatou component $U$ such that $f^m(U)\cap f^n(U)=\emptyset$ for $m\neq n$. This is essential, since the result does not hold for 
   \textit{eventually periodic} (i.e. non-wandering) Fatou components:  
     even some Jordan domains cannot arise as eventually periodic Fatou components of meromorphic functions (see Corollary~\ref{cor:not-realised}). Stronger restrictions are known for the geometry of certain types of periodic Fatou components \cite{azarina89,zdunik91,roesch-yin08};
        for example, any bounded immediate attracting basin of a polynomial is bounded either by a circle or by a Jordan curve of Hausdorff dimension greater than one. 
        
  Polynomials and rational maps do not have wandering domains by a famous theorem of Sullivan~\cite{sullivan85};
   hence the function $f$ in Theorem~\ref{thm:comp-fatou} must be transcendental. 
    Wandering domains of transcendental entire and meromorphic functions 
    have been the subject of intense study in recent years; see, e.g., 
  \cite{bishop15,benini-rippon-stallard16,bishop18,martipete-shishikura20,benini-evdoridou-fagella-rippon-stallard}. The first result about the realisation of 
  specific domains as wandering domains was proved by Boc Thaler~\cite{bocthaler21}.
   He showed that every bounded regular domain whose closure has connected complement is a 
   wandering domain of some transcendental entire function. As discussed below, Boc Thaler's result was extended to a larger 
   class of bounded domains in \cite{martipete-rempe-waterman22}, 
   whose closure may disconnect the plane. This class includes
   ``Lakes of Wada'' examples, where two or more such domains share the same boundary.

\begin{figure}
\begin{center}
\hspace{\fill}
\includegraphics[width=0.27\linewidth]{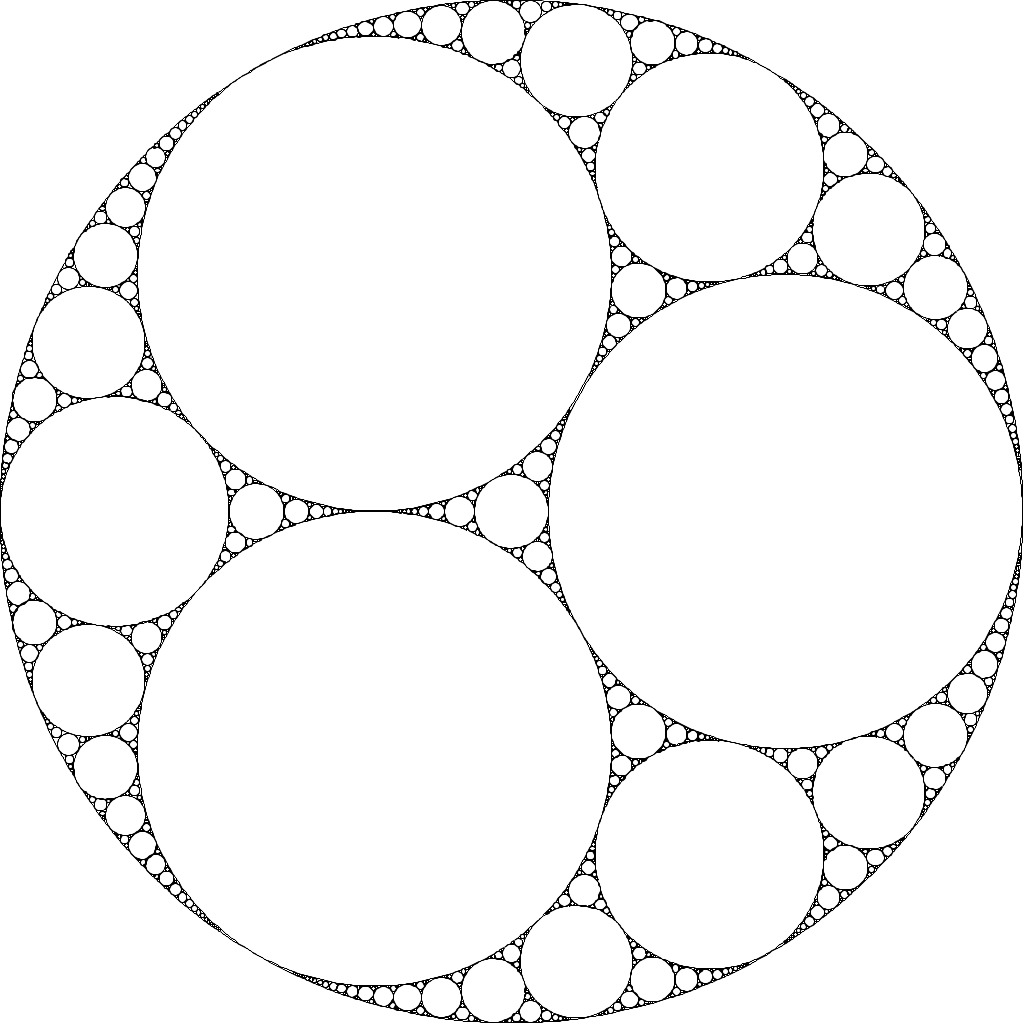} \hspace{\fill}
\includegraphics[width=0.27\linewidth]{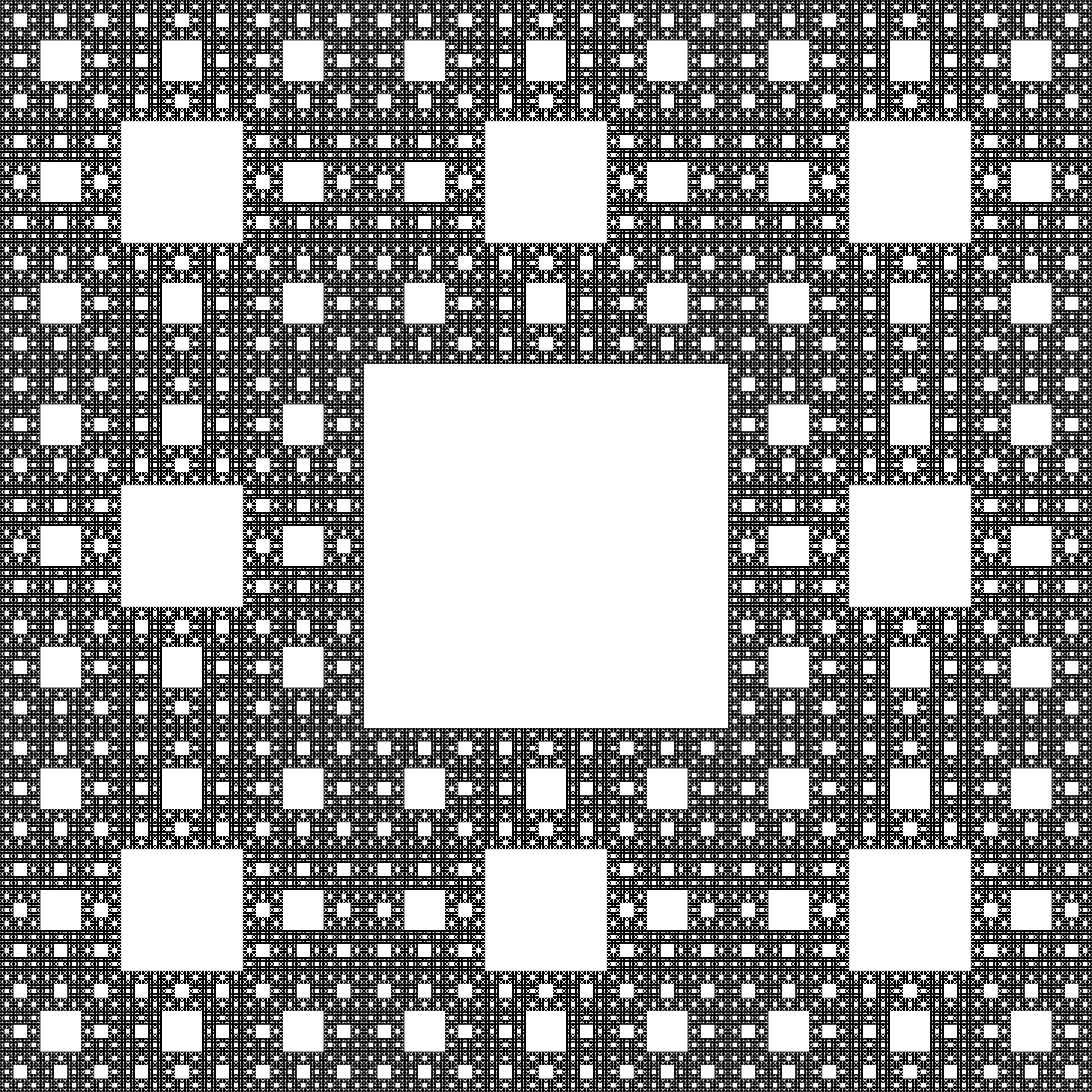} \hspace{\fill}~~
\includegraphics[height=0.27\linewidth]{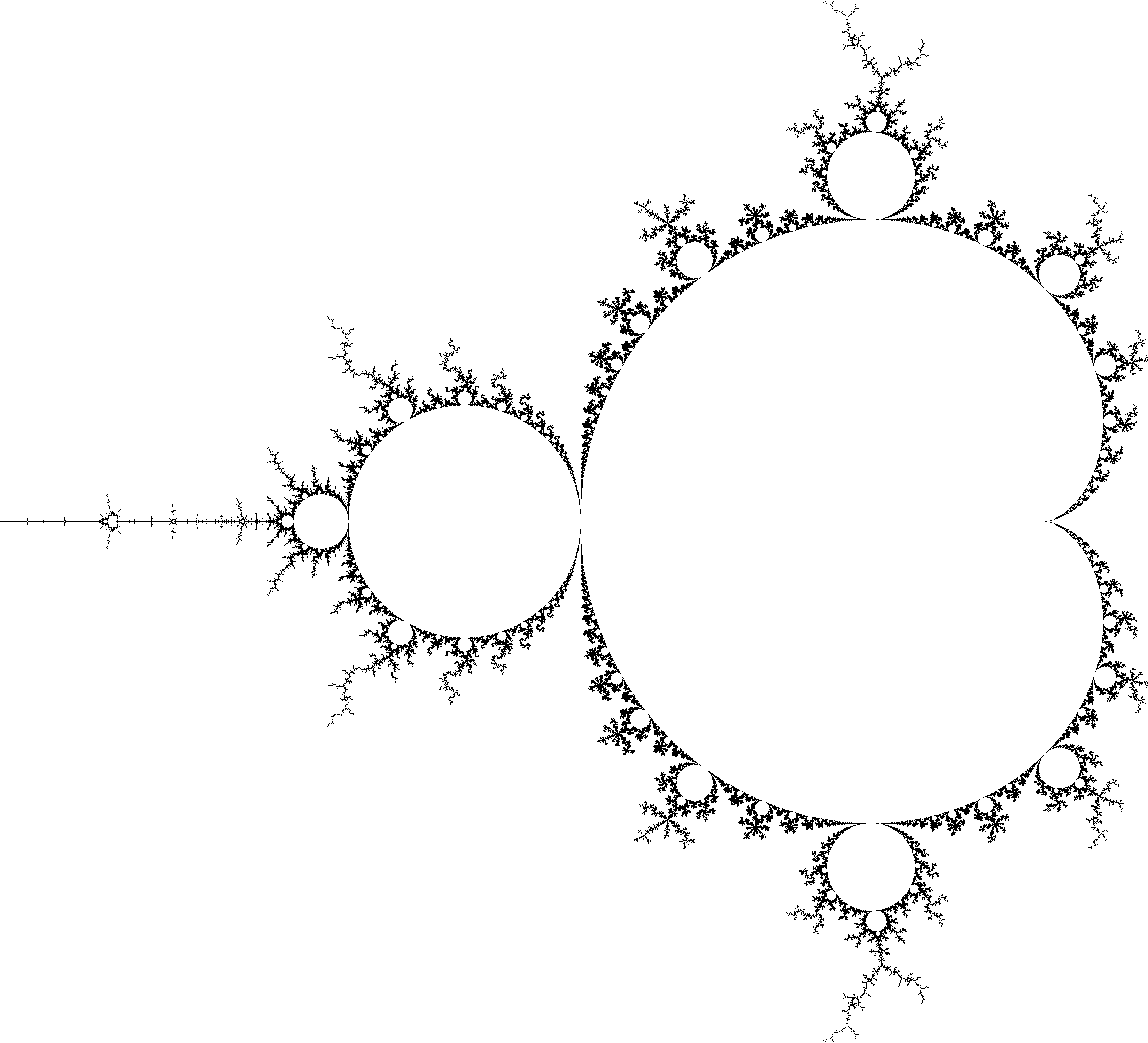}
\hspace{\fill}\,
\end{center}
\caption{By Theorem~\ref{thm:comp-julia}, each of the following continua with empty interior arise as a component of $J(f)$ for some meromorphic function $f$: the Apollonian gasket (left), the Sierpi\'nski carpet (centre), and the boundary of the Mandelbrot set (right).}
\label{fig:continuum}
\end{figure}

For a meromorphic function $f$, the \textit{escaping set} of $f$ is given by
   \[
  I(f)\defeq \bigl\{z\in\C \colon f^n(z) \textup{ is defined for all } n\in\N \textup{ and } f^n(z)\to\infty \textup{ as } n\to\infty\bigr\}.\]
We emphasise that pre-poles (points that are mapped to $\infty$ by a finite iterate of $f$) do not belong to $I(f)$. For transcendental 
entire functions, the escaping set was studied by Erermenko \cite{eremenko89}, who proved that $I(f)\cap J(f)\neq \emptyset$ and the components of $\overline{I(f)}$ are all unbounded. This definition was extended to transcendental meromorphic functions by Dom\'inguez \cite{dominguez98}, who showed that $I(f)\cap J(f)\neq \emptyset$ in this more general setting also.
However, components of the closure of the escaping set of such a function need not be unbounded \cite[p.~229]{dominguez98}.

Both Theorem~\ref{thm:comp-fatou} and \ref{thm:comp-julia} are proved using the following result (which also provides further examples of 
   bounded components of $\overline{I(f)}$).

\begin{thm}[Meromorphic functions with wandering compacta]
Let $K\subseteq \C$ be a compact set. Then there exists a transcendental meromorphic function $f$ such that\linebreak \mbox{$f_{\vert K}^n\to\infty$} uniformly as $n \to \infty$ and $f^n(K)\cap f^m(K) = \emptyset$ for $n\neq m$. Moreover:
\begin{enumerate}[(i)]
\item Every  component of $\partial K$ is a  component of $J(f)$. 
\item Every  component of $K$ is a component of $I(f)$ and of $\overline{I(f)}$. 
\item Every  component of $\interior(K)$ is a wandering domain of $f$.
\end{enumerate}
\label{thm:main}
\end{thm}

\subsection*{Comparison with entire functions} Theorem~\ref{thm:comp-fatou} does not hold if we require $f$ to be an \emph{entire} function. Indeed, any 
  multiply connected Fatou component $U$ of an entire function must be a bounded wandering domain on which the iterates of $f$ tend to infinity;
  more precisely, for sufficiently large $n$, $f^n(U)$ separates $f^{n-1}(U)$ from infinity \cite{baker84}(see also \cite{bergweiler-rippon-stallard13}). 
  It follows that $f^n|_U$ is not injective on $U$ for large $n$, which leads to restrictions on the possible geometry of $\partial U$ (see~Proposition~\ref{prop:ev-per}). 
  In particular, certain multiply connected domains cannot be wandering domains of an entire function, or more generally (using results of Rippon and Stallard~\cite{rippon-stallard08})
  of a meromorphic function with only a finite number of poles (see Corollary~\ref{cor:not-realised}). 
     In contrast, it was observed already by Baker, Kotus and L\"u~\cite{baker-kotus-lu-90} that multiply connected
      wandering domains of meromorphic functions with infinitely many poles can be much more varied. 
      In view of the above, our functions will be constructed to have infinitely many poles, and to be injective along the orbit of the wandering domains in question. As a consequence of 
      Theorem~\ref{thm:main}, we obtain a new proof of the following result from~\cite{baker-kotus-lu-90}: for every $1\leq n\leq\infty$, 
      there exists a transcendental meromorphic function
      that has a wandering Fatou component of connectivity $n$.
      
   Simply connected wandering domains of entire functions are more flexible than multiply connected ones. Let \mbox{$U\subset\C$} be a bounded simply connected  domain, and let
      $W$ be the unbounded connected component of $\C\setminus\overline{U}$. If $\partial W = \partial U$, which is a more general hypothesis than Boc Thaler's requirement~\cite{bocthaler21} that
       $W = \C\setminus\overline{U}$ (see Lemma~\ref{lem:regulardomains} and Figure~\ref{fig:topology}), 
       it is shown in~\cite{martipete-rempe-waterman22} that $U$ is a wandering domain of an entire function. 
       It is an open question whether there exists a simply connected wandering domain $U$ of an entire function
       such that $\partial W\subsetneq \partial U$ (see \cite[Question~1.16]{martipete-rempe-waterman22}). For rational
       and meromorphic functions, such domains may arise even as \emph{invariant} Fatou components: consider the map $z\mapsto 1/f(1/z)$ when $f$ is a 
       polynomial with connected Julia set and more than one bounded Fatou component.

   We are not aware of any analogue of Theorem~\ref{thm:comp-julia} for transcendental entire functions~$f$, or indeed of any 
       explicit non-degenerate continuum in the plane that can be realised as a Julia component of a transcendental entire function. For example, it is not known whether the unit circle
       can arise as such a component. If $J(f)$ has a bounded connected component $X$, then 
      $X$ is separated from any other point of $J(f)$ by a multiply connected wandering domain.

\subsection*{Remarks on the proof} 
The proof of Theorem~\ref{thm:main} follows a similar strategy as that of~\cite{bocthaler21} and~\cite{martipete-rempe-waterman22}: the function 
  $f$ is obtained as a limit of rational functions $f_n$ such that $f_n^n$ maps a suitable compact neighbourhood $K_n$ of $K$ univalently close to infinity. The function
  $f_{n+1}$ is obtained by modifying $f_n$ on $f_n^n(K_n)$, using an approximation theorem. In order to treat more general sets $K$ and 
  obtain stronger conclusions than in the entire setting, we introduce  two additional ideas. 
  On the one hand, we use a meromorphic version of Runge's theorem (Theorem~\ref{thm:Runge}), which allows us to place poles
  within the complementary components of $K$, and can therefore lead to multiply connected wandering domains. While each approximating rational function $f_n$ has
  only finitely many poles, the inductive approach ensures that any complementary component of $K$ will eventually be mapped over such a pole (even if there are infinitely
  many components). Secondly, we can ensure~-- again, by allowing for additional poles~--  that the whole boundary of each approximating set $K_n$ 
  is eventually mapped into 
  an attracting basin of our function $f$, crucially ensuring that $K$ is separated from every other point by some Fatou component.

We remark that it is possible to vary
 the construction to ensure different types of dynamical behaviour on the compact set $K$. For example,
  we could choose points in $K$ to be non-escaping, with orbits that are unbounded but have finite accumulation points. (Compare~\cite[Theorem~1.7]{martipete-rempe-waterman22} and \cite[Theorem~2]{bocthaler21}.) In the entire case, this accumulation set is necessarily infinite; for meromorphic functions, we may also 
   choose it to consist of a finite set of pre-poles. We omit the details.

\subsection*{Acknowledgements} We thank Misha Lyubich for asking a question that led to Theorem~\ref{thm:main} at a seminar given by the third author at Stony Brook University. We also thank Mitsuhiro Shishikura for a discussion about this topic at an online seminar given by the first author at Kyoto University. We are grateful to Chris Bishop and Daniel Meyer
for interesting discussions concerning Proposition~\ref{prop:weirddomains}. Finally, we thank the referee for helpful comments that improved the presentation of this paper.

\section{Topology of Fatou components}\label{sec:topology}

In this section, we discuss the topological properties of Fatou components of meromorphic functions.
We begin with a preliminary observation.

\begin{lem}[Regular domains] \label{lem:regular-def}
A domain $U$ is regular if and only if $\partial U=\partial \overline{U}$.
\end{lem}
\begin{proof}
Recall that a domain $U$ is regular if and only if $U = \textup{int}\,\overline{U}$. On the one hand, we have $\operatorname{int} \overline{U}=\overline{U}\setminus \partial \overline{U}$. On the other hand, since $U$ is open, 
  $U = \overline{U}\setminus \partial U$. Hence $U = \operatorname{int} \overline{U}$ if and only if $\partial U = \partial \overline{U}$. \end{proof}

Boc Thaler \cite[p.~2]{bocthaler21} 
   observed that wandering domains of transcendental entire functions are regular 
    domains. To prove that \ref{item:Fatoucomponent} implies \ref{item:Uregular} in 
   Theorem~\ref{thm:comp-fatou},
    we note that more generally \emph{any} (not necessarily bounded, or wandering) 
    Fatou component of a meromorphic function $f$
    is regular, unless $F(f)$ is connected.
      Since we are not aware of a reference for this more general statement, we provide the proof here.
                           
\begin{lem}[Fatou components are regular] \label{lem:regular-fatou-comp}
  Let $f$ be a meromorphic function and suppose that $F(f)$ is disconnected. Then every Fatou component of $f$
 is a regular domain. 
\end{lem}
\begin{proof}
 Let $U$ be a Fatou component of $f$; the hypothesis implies that $F(f)\setminus U\neq \emptyset$.
   If $U$ is periodic, let $n$ be its period and let $V$ be any other Fatou component of $f$.
   Otherwise, set $V=f^{-1}(U)$ and $n=1$. In either case, 
      $f^{nj}(\overline{U})\cap V = \emptyset$ for all $j\geq 0$. In other words, $\bigcup_{j=0}^{\infty} f^{-nj}(V) \subset  \C\setminus \overline{U}$. 
      
      Since $V$ is a non-empty open set,  we have 
     \[ J(f) = J(f^n) \subseteq \overline{\bigcup_{j=0}^{\infty} f^{-nj}(V)} \subset \overline{  \C\setminus \overline{U}}\]
     by the blowing-up property of the Julia set for $f^n$~\cite[Theorem~A(5)]{bakerdominguezherring}). 
    So 
       \[ \partial U \subset J(f)\subset \overline{\C\setminus \overline{U}}. \] 
     Thus $\partial U = \partial \overline{U}$ and $U$ is regular  by Lemma~\ref{lem:regular-def}. 
\end{proof}

\begin{figure}
\centering
\captionsetup[subfigure]{justification=centering}
\vspace*{-0.4cm}
\begin{subfigure}{0.48\textwidth} \centering \includegraphics[width=.8\linewidth]{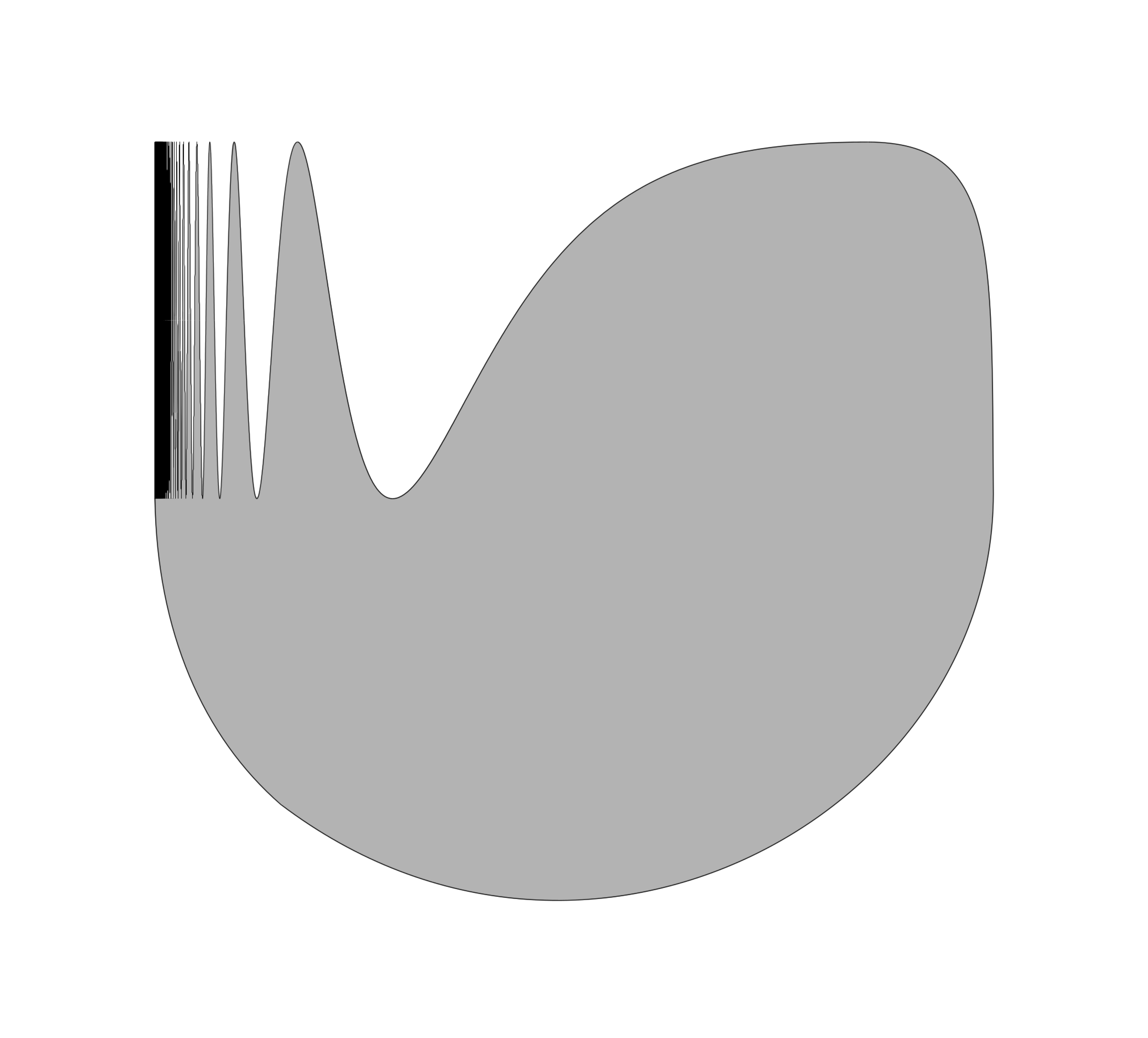} \vspace{-0.1cm} \caption{ $U_1$: interior of a Warsaw circle \\ $U_1$ is regular\\ $\partial U_1 = \partial \Fill(\overline{U_1})$\\ $\C\setminus \overline{U_1}$ is connected}
 \end{subfigure}
\begin{subfigure} {0.48\textwidth} \centering \includegraphics[width=0.8\linewidth]{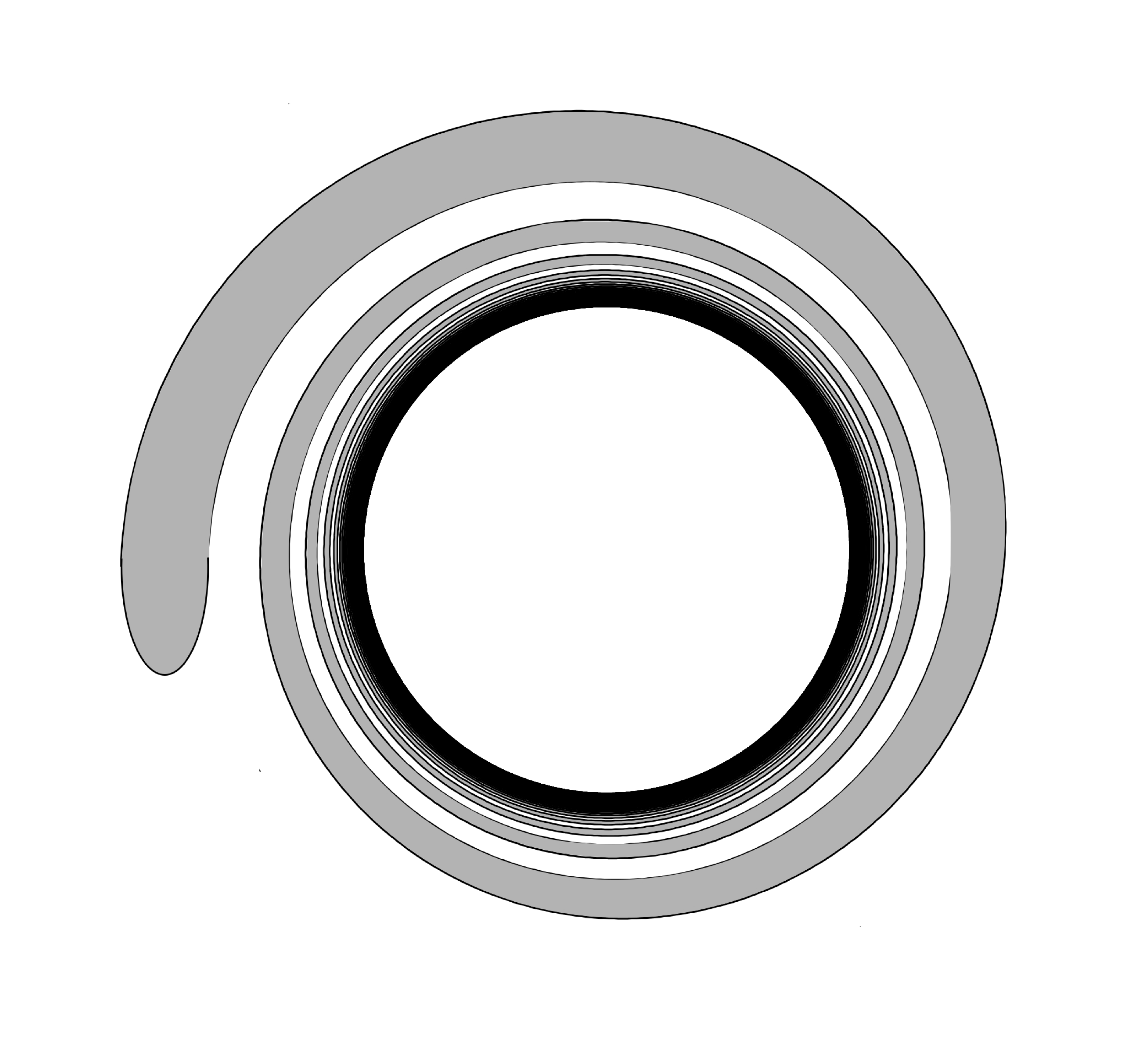} \vspace{-0.1cm} \caption{$U_2$: inwards spiral\\ $U_2$ is regular\\ $\partial U_2 = \partial \Fill(\overline{U_2})$\\ $\C\setminus \overline{U_2}$ is disconnected}\end{subfigure}
 \\
 \vspace{0.2cm}
 \begin{subfigure}{0.48\textwidth} \centering \includegraphics[width=0.8\linewidth]{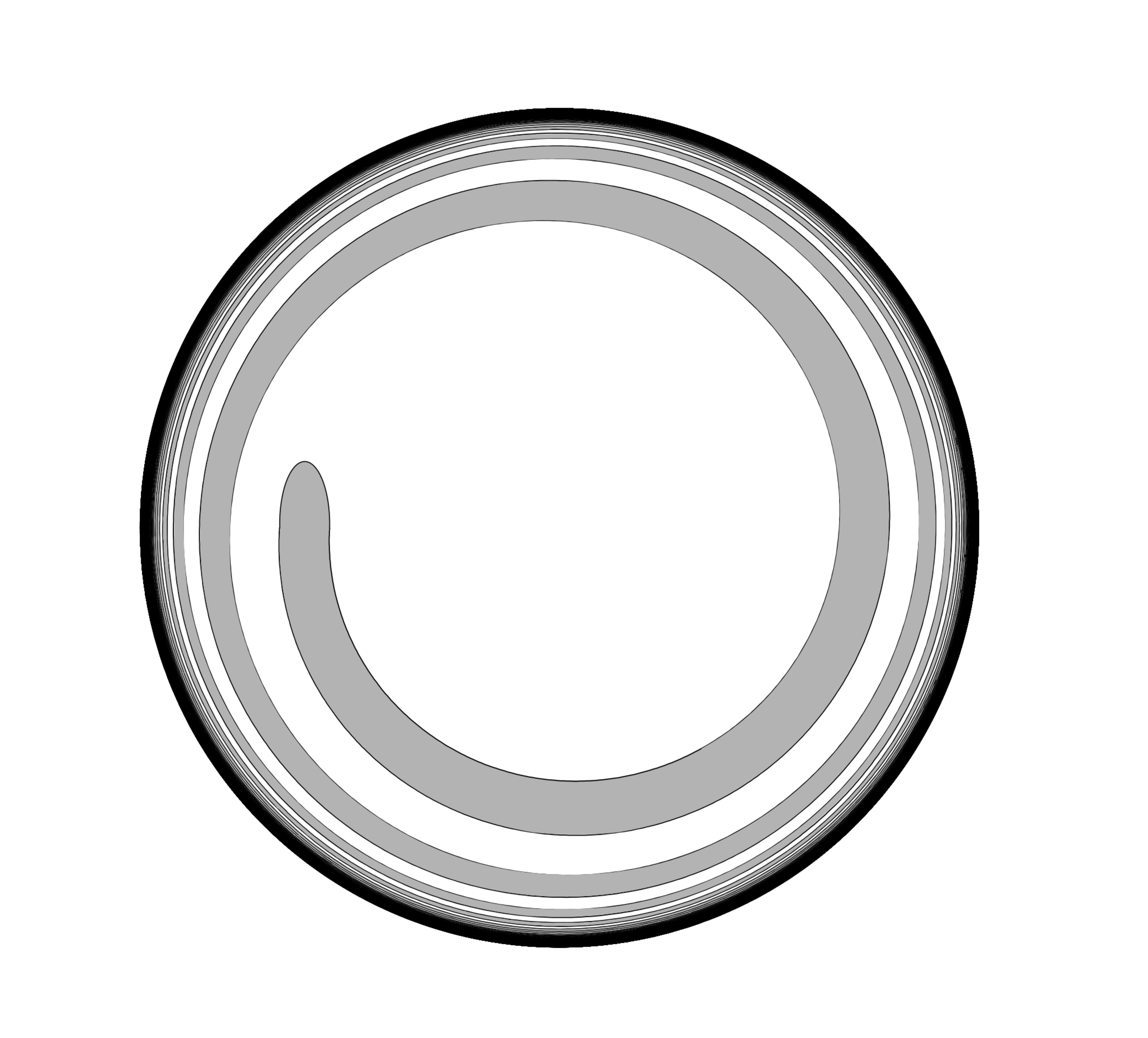} \caption{$U_3$: outwards spiral\\ $U_3$ is regular\\ $\partial U_3 \supsetneq \partial \Fill(\overline{U_3})$\\ $\C\setminus \overline{U_3}$ is disconnected}
  \end{subfigure}
  \begin{subfigure}{0.48\textwidth} \centering \includegraphics[width=.8\linewidth]{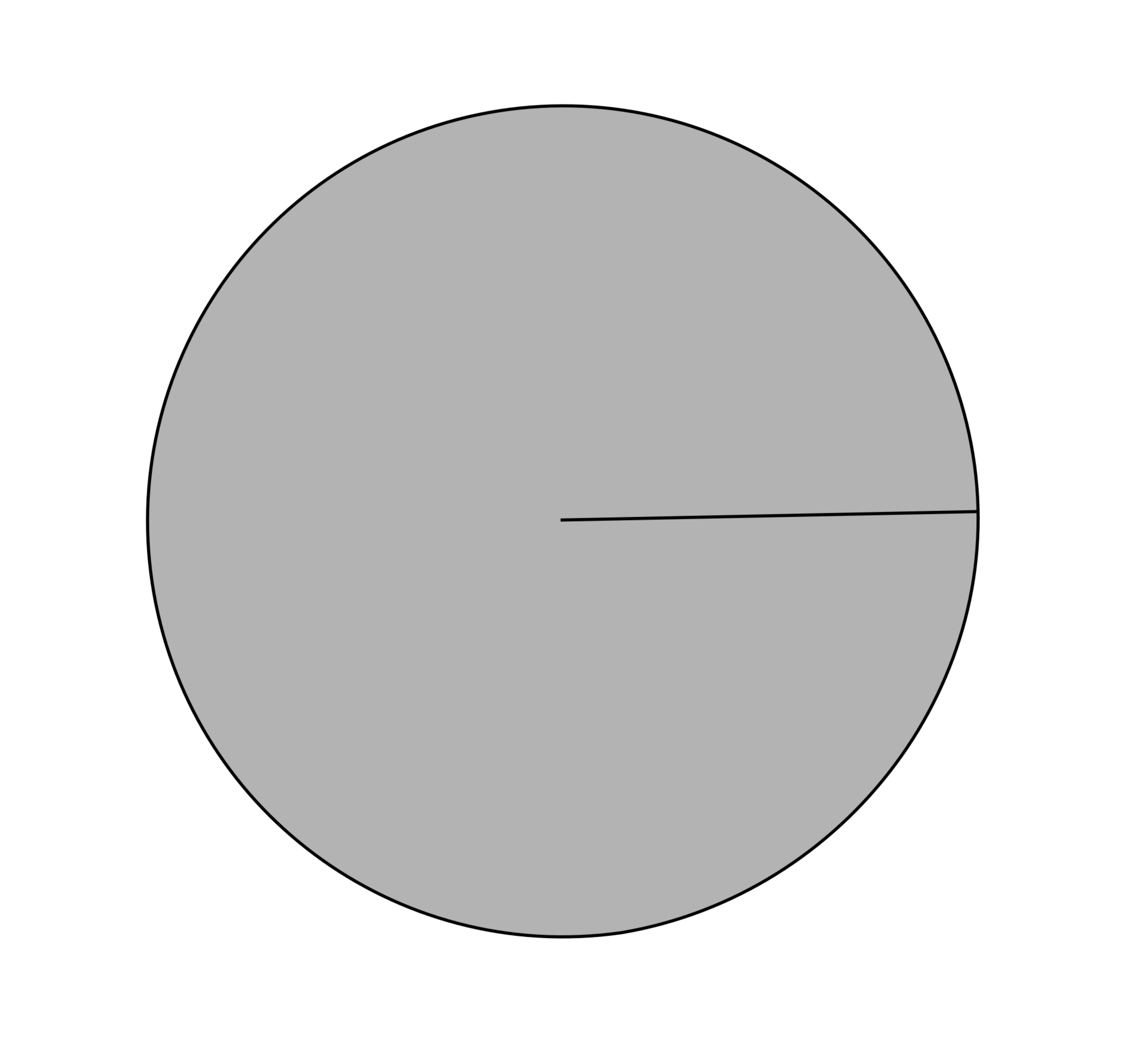}  
\caption{$U_4$: slit disc\\ $U_4$ is not regular\\ $\partial U_4 \supsetneq \partial \Fill(\overline{U_4})$\\ $\C\setminus\overline{U_4}$ is connected}
\end{subfigure}
 \caption{Four examples $U_1,\dots, U_4$ of simply connected domains.
    The first three may arise as Fatou components of meromorphic functions by Theorem~\ref{thm:comp-fatou}, while $U_4$ does not by Lemma~\ref{lem:regular-fatou-comp}. 
   For $U_1$ and $U_2$, the function $f$ may be chosen to be entire by~\cite{bocthaler21} and \cite{martipete-rempe-waterman22}, respectively. It is an open   
    question whether or not $U_3$ can be realised as a Fatou component of an entire function.}\label{fig:topology}
\end{figure}

We now compare the different topological notions for regular simply connected domains discussed in the introduction. 
  If $K\subset\C$ is compact, and $W$ is the unbounded connected component of
   $\C\setminus K$, then $\Fill(K) \defeq \C\setminus W$ is called the \emph{fill} of $K$. 

\begin{lem}[Bounded simply connected domains]\label{lem:regulardomains}
 Consider the following properties of a bounded simply connected domain $U$. 
  \begin{enumerate}[(1)]
    \item $U$ is regular and $\C\setminus \overline{U}$ is connected.\label{item:bocthaler}
    \item $\partial U = \partial \Fill(\overline{U})$.\label{item:fillboundary}
    \item $U$ is regular.\label{item:Uregulartopology}
  \end{enumerate}
   Then \ref{item:bocthaler} $\Rightarrow$~\ref{item:fillboundary} $\Rightarrow$~\ref{item:Uregulartopology}. Moreover, the converse of neither implication 
   holds in general. 
\end{lem}
\begin{proof}
   If \ref{item:bocthaler} holds, then by Lemma~\ref{lem:regular-def}, $\partial U = \partial \overline{U} = \partial(\C\setminus \overline{U}) = \partial \Fill(\overline{U})$. 
     If~\ref{item:fillboundary} holds, then $\partial \overline{U} \subset \partial U = \partial \Fill(\overline{U}) \subset \partial \overline{U}$, so both inclusions are equalities, and 
     $U$ is regular by Lemma~\ref{lem:regular-def}. For the final claim, see the examples in Figure~\ref{fig:topology}. 
\end{proof}

\section{Domains that cannot be eventually periodic Fatou components}

In this section, we show that certain 
     regular~-- even Jordan~-- 
     domains cannot arise as \emph{eventually periodic} Fatou components
     of a meromorphic function $f$, or even as wandering domains on which
     the iterates are not univalent. 
     Indeed, for such a domain there must be different points on the boundary
     that are related by a combination of 
     locally univalent forward and backward iterates of $f$ (see Proposition~\ref{prop:ev-per}). 
     So a domain whose boundary
     near any given point is not biholomorphically equivalent to the boundary
     near any other cannot be realised as a Fatou component of this type.
     
To provide the details of this argument, we use a version of the Gross star theorem for iterates
     of meromorphic functions, which follows from the general form of the Gross theorem given by Kaplan~\cite[Theorem~3]{kaplan}.
     
\begin{thm}[Gross star theorem for iterates]\label{thm:gross}
 Let $g\colon\C\to\Ch$ be a meromorphic function, and let 
      $z_0\in\C$ and $k\geq 1$ be 
      such that $w_0\defeq g^k(z_0)$ is defined and $(g^k)'(z_0)\neq 0$. 
      
   Let $W\ni w_0$ be a simply connected domain; for $\theta\in (0,2\pi)$, let
     $\gamma_{\theta}$ denote the hyperbolic geodesic of $W$ 
     starting at $w_0$ in the direction $\theta$. Then, for almost every 
     $\theta\in (0,2\pi)$, $\gamma_{\theta}$ is an arc connecting $w_0$ to an end-point $\omega_\theta \in \partial W$, and the branch $\beta$ of $g^{-k}$ that maps $w_0$ to $z_0$
     can be analytically continued along $\gamma_{\theta}$ into $\omega_\theta$.
\end{thm}
\begin{proof}
The claim follows from~\cite[Theorem~3]{kaplan} as follows. Adopting the notation of~\cite[Section~3]{kaplan}, we set
  $\phi = g^k$ and let $D=g^{-(k-1)}(\C) = \phi^{-1}(\Ch)$ be its domain of definition; note that $\Ch\setminus D$ is countable. Let $\eta \colon \D\to W$ be
  a conformal isomorphism with $\eta(0)=w_0$ and $\eta'(0)>0$, and let $\eps\in(0,1/2)$ be sufficiently small that
    the branch $\beta$ is defined on $\eta(D(0,2\eps))\subset W$.
   Then we take 
  Kaplan's $f$ to be 
     \[ f(\sigma) = \eta\bigl(e^{\log\eps -2\pi i \sigma }\bigr),\] 
  where $\sigma  = s + it$ with $0<s<1$ and $-\infty < t < h_{s} \defeq h \defeq -\log\eps$. 
  For fixed $s\in (0,1)$, the curve $\gamma_{-2\pi s}\colon t\mapsto f(s + it)$ is a parameterisation of 
   the geodesic $\gamma_{-2\pi s}$ of $W$. Let $h_1(s) \geq b\defeq \log 2$ be maximal such that 
  $\beta$ can be analytically continued along $\gamma_{-2\pi s}$ for $t<h_1(s)$. 
  
  Now~\cite[Theorem~3]{kaplan} implies that $h_1(s) = h$ for almost all $s$; that is, $\beta$ can be continued analytically
  along the entire geodesic $\gamma_{-2\pi s}$. 
   By Fatou's theorem, $\gamma_{-2\pi s}(t)$ also has a limit point 
  $\omega_{-2\pi s} \in \partial W$ as $t\to h$, for almost every $s$ (see~\cite[Theorem~17.4]{milnor06}). If $s$ belongs to the set of full measure for which both properties hold, then either
  $\beta$ can be analytically continued into $\omega_{-2\pi s}$, or 
  $\beta(\gamma_{-2\pi s}(t))$ tends to a point of $D$ or to a critical point of $f^k$ as $t\to h$. Since the set of such points is countable,
  it follows from~\cite[Theorem~1]{kaplan}  (as in the proof of~\cite[Theorem~3]{kaplan}) that the set of $s$ with the latter property has measure zero, completing the proof.
\end{proof}

\begin{prop}[Locally biholomorphic boundary] \label{prop:ev-per}
Let $f$ be a meromorphic function and suppose $U$ is a (bounded or unbounded) Fatou component of~$f$ whose boundary is a 
  finite union of Jordan curves. 
  Suppose furthermore that $U$ is either eventually periodic, or that there is some
    $k\geq 1$ such that $f^k|_U$ is not injective. 
  
   Then there are $\zeta_1 ,\zeta_2\in \partial U$, $\zeta_1\neq \zeta_2$, neighbourhoods $V_1,V_2$ of
      $\zeta_1$ and $\zeta_2$, respectively, and a biholomorphic map $\psi\colon V_1\to V_2$ such that $\psi(V_1\cap \partial U) = V_2\cap \partial U$ and $\psi(\zeta_1) = \psi(\zeta_2)$.
\end{prop}
\begin{proof}
 By assumption, there is $k\in\N$ such that either the Fatou component $U_k$ containing $f^k(U)$ is periodic (that is, there is $p\in\N$ such that $f^p(U_k)\subseteq U_k$),
        or such that $f^k\colon U\to U_k$ is not injective. 
      Choose some point $z_0\in \partial U$ that is not 
      on the grand orbit of $\infty$ or of any critical or periodic point. (The hypothesis implies that $f$ is neither
      constant nor a M\"obius transformation, so the latter set is countable. On the other hand, $\partial U$ is uncountable by hypothesis, so a point $z_0$ with
      the desired properties exists.) 
      In particular, $(f^k)'(z_0)\neq 0$. 
      Choose a small Jordan neighbourhood $V_0$ of 
      $z_0$ whose boundary intersects $\partial U$ in exactly two points and on
      which $\phi\defeq f^k|_{V_0}$ is injective.
      Then $V_0\cap \partial U$ is an arc, as is $\alpha \defeq f^k(V_0\cap \partial U)\subset \partial U_k$.
        (Note that we do not claim that $U_k$ is necessarily a Jordan 
      domain, nor that $f^k(V_0\cap \partial U) = f^k(V_0)\cap \partial U_k$.) 
      Recall that $f^k(z_0)$ is not periodic by choice of $z_0$.
      Hence, if $U_k$ is periodic of period $p$, then by shriking $V_0$ if necessary 
      we may also assume that
       $f^p(\alpha)\cap \alpha = \emptyset$.
      
  Every point
       $w \in W\defeq f^k(V_0\cap U)\subset U_k$ has 
       exactly one preimage in $V_0$ under $f^{k}$, namely 
       $\phi^{-1}(w)$. Set $\tilde{k}\defeq k$ if $f^k\colon U\to U_k$ is
          not univalent, and $\tilde{k}\defeq k+p$ otherwise, where $p$ is the
          period of the Fatou component $U_k$. 
          We claim that there is $w_0\in W$ 
          that has a simple preimage $z_0\in U\setminus\{\phi^{-1}(w_0)\}$ 
          under $f^{\tilde{k}}\colon U\to U_k$.

     We use a theorem of Bolsch~\cite[Theorem 1]{bolsch}: either there is a finite $d\geq 1$ such that 
       $f^{\tilde{k}}\colon U\to U_k$ takes every value in $U_k$ exactly $d$
       times (counting multiplicity), or $f^{\tilde{k}}\colon U\to U_k$ 
       takes every value in $U_k$ 
       infinitely often, with at most two exceptions, in which case we set $d=\infty$. 
       Let $w_0 \in W$ be a point that is not on the   
       grand orbit of a critical 
       or periodic point, and not an exceptional point for
       Bolsch's theorem (in the case where $d=\infty$). 
       Then $f^{\tilde{k}}|_U$ takes the value $w_0$ exactly $d$ times. 
       (Since $w_0$ is not on a critical orbit, it has no
       multiple preimages under $f^{\tilde{k}}$.)

       If $d > 1$, then we may pick $z_0\in U\cap f^{-\tilde{k}}(w_0)$ with
       $z_0\neq \phi^{-1}(w_0)$. Otherwise, let $z_0$ be
       the unique element of $f^{-\tilde{k}}(w_0)$. 
     Let $\tilde{\phi}$ be a univalent restriction of $f^{\tilde{k}}$ near
       $z_0$. By Theorem~\ref{thm:gross}, there is a hyperbolic geodesic
       $\gamma$ of $W$, starting at $w_0$
       and ending at a point $\omega\in\alpha$, along which the branch $\beta = \tilde{\phi}^{-1}$ can
       be continued analytically.

    Let $\zeta_1 = \beta(\omega)\in \partial U$ be the point
     obtained by analytic continuation of $\beta$ along $\gamma$,
     and set $\zeta_2 \defeq \phi^{-1}(\omega)$. 
     If $k=\tilde{k}$, then $d>1$ and hence $z_0\neq \phi^{-1}(w_0)$. Since
     $\gamma\subset \phi(V_0)$, it follows that $\zeta_1\neq \zeta_2$. 
     Otherwise, $f^{k+p}(\zeta_1) = \omega \in \alpha$ and $f^{k+p}(\zeta_2) = f^p(\omega) \in f^p(\alpha)$.
     Since we chose $\alpha$ small enough to ensure that
     $f^p(\alpha)\cap \alpha = \emptyset$, also $\zeta_1\neq \zeta_2$ in this case. 

   The map $\psi\defeq \phi^{-1}\circ f^{\tilde{k}}$ is univalent on a neighbourhood $V_1$ of $\zeta_1$
     and maps $\zeta_1$ to $\zeta_2$. We may assume that $V_1$ is chosen such that $V_1\cap \partial U$ is  an arc; then $\psi(V_1\cap \partial U)$ is also an arc, and therefore
     a relative neighbourhood of $\zeta_2$ in $\partial U$. Shrinking $V_1$ further, if necessary, we can ensure that $\psi(V_1\cap \partial U) =  V_2\cap \partial U$, where
     $V_2 = \psi(V_1)$. The proof is complete.
\end{proof}

To conclude that there exist domains that cannot be realised as eventually periodic Fatou components, or as multiply connected Fatou components 
  of meromorphic functions with only finitely many poles, we use the following observation.
  
\begin{prop}[Domains with everywhere inhomogeneous boundaries]\label{prop:weirddomains}
  For every $k\geq 1$, there exists a bounded domain $U\subset\C$, bounded by
     $k$ disjoint Jordan curves, such that 
    $\partial U$ is not locally biholomorphic near any two distinct points $\zeta_1,\zeta_2\in \partial U$. 
    That is, there does not exist a biholomorphic map $\psi\colon V_1\to V_2$ between any neighbourhoods of $V_1,V_2$ of
      $\zeta_1$ and $\zeta_2$, respectively, such that $\psi(V_1\cap \partial U) = V_2\cap \partial U$ and $\psi(\zeta_1) = \zeta_2$.
\end{prop}
\begin{proof}
Such examples can be constructed by many methods. Since we are not aware of a reference, we outline one such construction:
  take boundary curves $\gamma_1,\dots,\gamma_k$ 
    with the property 
     that the local Hausdorff dimension of $\partial U = \bigcup_{j=1}^k \gamma_k$ is different at any two distinct points of
     $\partial U$. (Here by the local Hausdorff dimension at a point $z$ we mean the infimum of the 
     Hausdorff dimension over all relative neighbourhoods of $z$ in $\partial U$.)
     Since  Hausdorff dimension is preserved by bi-Lipschitz maps, it follows
     that $U$ has the desired property. 
     
    To show that such curves exist, consider 
 a generalisation of the standard
  von Koch curve. Recall that the von Koch curve is obtained as a limit of curves $\gamma_n\subset\C$, with $\gamma_0 = [0,1]$, and 
  $\gamma_n$ a polygonal arc consisting of $4^n$ line segments. To obtain $\gamma_{n+1}$ from $\gamma_n$, subdivide each arc of
  $\gamma_{n}$ into three equal parts, and replace each middle interval by two arcs 
   that form an equilateral triangle
  with the deleted arc. (See e.g.~\cite[Figure~0.2(a)]{falconerfractalgeometry}.)
  
 In other words, the von Koch curve is the limit set of an iterated function system consisting of four similarities with contraction factor of $1/3$, leading to 
  a curve of Hausdorff dimension $\log 4/\log 3$. This construction may be generalised by replacing each interval instead by four intervals of length
    $1/s$, where $2 < s \leq 4$. All of these systems satisfy the open set condition for 
    the right-angled isosceles triangle $T$ with vertices $0$, $1$ and $1/2 + 1/2 i$. 
    By the usual formula for the Hausdorff dimension of self-similar fractals~\cite[Theorem~9.3]{falconerfractalgeometry}, 
    the resulting limit set~-- which is an arc connecting $0$ to $1$ within $\overline{T}$~-- 
    has Hausdorff dimension $h = \log 4 / \log s$. 
    For $s=4$, this curve coincides with the interval $[0,1]$, while for $s=2$ (which we omitted to avoid self-intersections), 
    the construction leads to a space-filling curve with image $\overline{T}$. 
    
  For any closed interval $[h_1,h_2]\subset [1,2]$, we modify this construction as follows, to obtain an arc along which the local Hausdorff dimension is continuously strictly increasing
    from $h_1$ to $h_2$. Set $s_1 \defeq 4^{1/h_1}$ and
    $s_2 \defeq 4^{1/h_2}$. We 
    follow the construction of the von Koch curve, but at the $n$-th step we replace the $j$-th linear
    segment of $\gamma_n$ ($0\leq j < 4^n$) with
    four segments of length $1/s_{n,j}$ where  
      \[ s_{n,j} = s_1 - \frac{j}{4^n}\cdot (s_1-s_2). \] 
   Parameterise the resulting curves as $\gamma_n\colon [0,1]\to \C$ in such a way that the preimage of each segment of $\gamma_n$ has length $4^{-n}$.
    Then $\gamma_n$ converges uniformly to an arc \mbox{$\gamma\colon [0,1]\to\C$}. Set
       $s(t) \defeq s_1 + t(s_1 - s_2)$; then it follows (by the same argument as for \cite[Theorem~9.3]{falconerfractalgeometry})
        that the Hausdorff dimension of $\gamma( [t_1,t_2])$ is bounded from
       below by $\log 4 / \log s(t_1)$ and from above by $\log 4 / \log s(t_2)$. Hence the local
       Hausdorff dimension of $\gamma([0,1])$ at a point $t$ is given by $\log 4 / \log s(t)$, and the arc $\gamma([0,1])$ has the desired property. 
       
   Now apply this construction to $3k$ pairwise disjoint closed subintervals of $[1,2]$. Using affine transformations of the plane,
     we may arrange these arcs into $k$ ``snowflakes''
     (compare \cite[Figure~0.2(b)]{falconerfractalgeometry}) that
     form the boundary of a domain $U$, as desired. 
\end{proof}
\begin{rmk}
  In our proof of Proposition~\ref{prop:weirddomains}, if none of the subintervals of [1,2] that are used 
     contain the value $h=2$, then the boundary curves constructed are quasicircles. It is plausible 
   that the boundary curves of the
   domain $U$ in Proposition~\ref{prop:weirddomains} can instead be chosen to have any level of smoothness 
   short of being analytic. 
\end{rmk} 

\begin{cor}[Non-realisability by eventually periodic Fatou components]\label{cor:not-realised}
If $U$ is a domain as in Proposition~\ref{prop:weirddomains}, and $U$ is a Fatou component of a meromorphic function~$f$, then
   $f$ is transcendental, $U$ is a wandering domain and $f^n|_U$ is injective for all $n\geq 1$.    
   In particular, if $U$ is multiply connected, then $f$ has infinitely many poles.
\end{cor}
\begin{proof}
By Proposition~\ref{prop:ev-per}, $U$ is a wandering domain and $f^n|_U$ is injective for all $n\geq 1$; by Sullivan's theorem, it follows that $f$ is transcendental.

   If $f$ has finitely many poles and $U$ is a multiply connected Fatou component, then by~\cite[Theorem~1(a)]{rippon-stallard08} either the Fatou component 
      $U_k$ containing $f^n(U)$ is simply connected for all sufficiently large $n$, or $U_n$ is multiply connected and surrounds $U_{n-1}$ for all sufficiently large 
       $n$. In either case, $f^n|_U$ cannot be injective for large $n$, and the final claim of the corollary follows from the first.
\end{proof}

\section{Approximation}
 To construct the function in Theorem~\ref{thm:main}, we will use the following classical result of approximation theory (see \cite[Theorem~2~on~p.~94]{gaier87}). Note that in    
   \cite[Section~3]{martipete-rempe-waterman22}, we use a 
   version of Runge's theorem where $\C\setminus A$ is connected and the approximation can be carried out by polynomials, 
   while here the approximation is by rational maps and $A$ can be any compact set. 
   (The theorem for polynomials follows from that for rational maps by moving the poles; we refer to~\cite{gaier87} for details.) 

\begin{thm}[Runge's theorem]
 \label{thm:Runge}
  Let $A\subseteq \C$ be a compact set. 
 Suppose that $g$ is meromorphic on an open neighbourhood of $A$. 
  Then for every $\eps>0$, there exists a rational function~$f$ such that 
\[
\lvert f(z) - g(z)\rvert < \eps\quad \text{for all } z\in A\setminus g^{-1}(\infty).
\] 
\end{thm}
\begin{rmk}
  The theorem is often stated only in the case where $g$ has no poles on~$A$, and hence analytic on an 
    open neighbourhood of $A$. This implies the stated result. Indeed, $g$ has only finitely many poles on $A$, and hence
    we may find a rational function $R$ such that $\tilde{g}\defeq g-R$ is analytic on a neighbourhood of $A$. 
    Apply the aforementioned special case to $\tilde{g}$ to obtain a rational function $\tilde{f}$ whose poles are
    outside $A$; then $f\defeq \tilde{f}+R$ is the desired approximation of $g$. 
\end{rmk}

We also require the following elementary result on approximation, which is a special case 
  of~\cite[Corollary~2.7]{martipete-rempe-waterman22}. 
\begin{lem}[Approximation of univalent iterates]\label{lem:approx}
   Let $U\subseteq\C$ be open, and $g\colon U\to\C$ be holomorphic. Suppose that $G\subset U$ is open and such that
   $g^n$ is defined and univalent on $G$ for some $n\geq 1$.
   
   Then for every compact 
       $K\subset G$ and every $\eps>0$, there is a $\delta>0$ with the following property.
       For every holomorphic function $f\colon U\to\C$ with 
\[
       \lvert f(z) - g(z)\rvert \leq \delta \quad \text{for all } z\in U,
       \] 
       the function $f^n$ is defined and injective on $K$, with 
\[\lvert f^k(z) - g^k(z)\rvert \leq \eps\quad \text{for all } z\in K \text{ and } 1\leq k \leq n. 
       \] 
 \end{lem}

\section{Meromorphic functions with wandering compacta}

Let $K\subseteq \C$ be a compact set. By applying an affine transformation, we may assume without
  loss of generality that $K\subseteq \DD$. 
  For $j\geq 0$, define 
\[
K_j:=\bigl\{z\in\C \colon \textup{dist}(z,K)\leq 2^{-(j+1)}\cdot \dist(K,\partial\DD) \bigr\}. 
\]

  Then $(K_j)_{j=0}^{\infty}$ is a sequence of compact sets with the following properties: 
   \begin{enumerate}[(a)]
     \item $K_{0} \subseteq \DD$.
     \item $K_j\subseteq \interior(K_{j-1})$ for all $j\geq 1$.
     \item $\bigcap_{j=0}^{\infty} K_j = K$.
   \end{enumerate}
         
 \begin{figure}
\begin{center}
\vspace{5pt}
\includegraphics[height=0.3\linewidth]{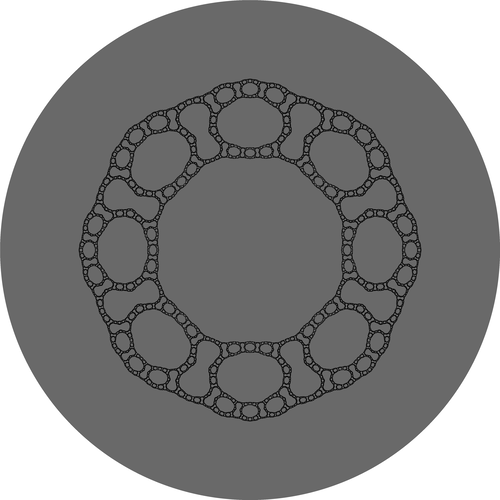}\quad
\includegraphics[height=0.3\linewidth]{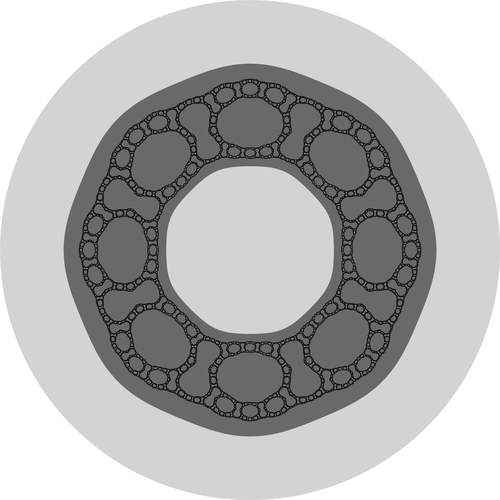}\quad
\includegraphics[height=0.3\linewidth]{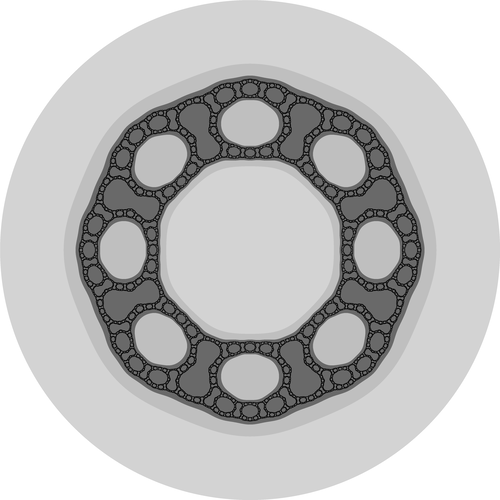}\vspace{10pt}\\
\includegraphics[height=0.3\linewidth]{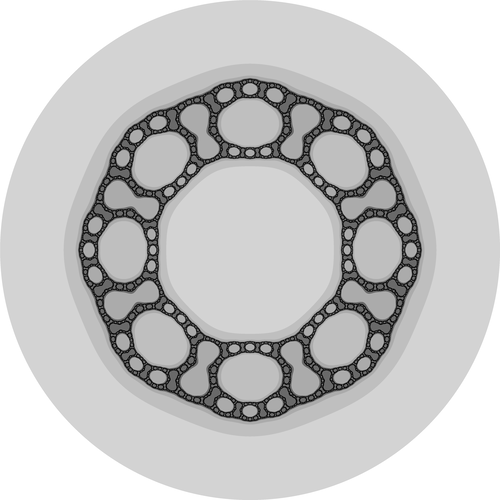}\quad
\includegraphics[height=0.3\linewidth]{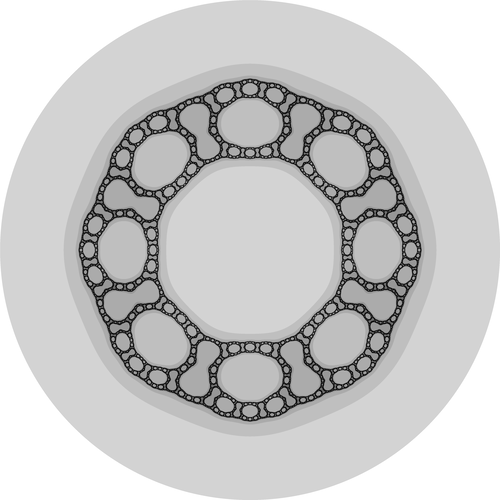}\vspace{10pt}
\put (-292,160) {$K_0$}
\put (-145,160) {$K_1$}
\put (1,160) {$K_2$}
\put (-217,-20) {$K_3$}
\put (-73,-20) {$K_4$}
\end{center}
\caption{The first sets in a nested sequence of compacta $(K_j)$ shrinking down to the Julia set of the rational map $f(z)=z^4-0.1/z^4$, which is a Sierpi\'nski curve; see \cite[Figure~2]{devaney13}. (Compare to \cite[Figure~3]{martipete-rempe-waterman22}.)}
\label{fig:devaney}
\end{figure}

For $j\geq -1$, we define the discs
    \[ D_j \defeq D(3j,1)=\bigl\{z\in\C\colon |z-3j|<1\bigr\}. \]
    Our goal is to construct a meromorphic function $f$ for which $f^j(K)$ is in $D_j$ for all $j\geq 0$, while the boundary
     of every $K_j$ is eventually mapped to an attracting basin containing the disc $D_{-1}$. 
     
     \begin{prop}\label{prop:main}
       There is a transcendental meromorphic function $f$ with the following properties:
        \begin{enumerate}[(a)]
          \item $f(\overline{D_{-1}})\subseteq D_{-1}$.
          \item $f^j$ is injective on $K_j$ for all $j\geq 0$, with
               $f^j(K_{j})\subseteq D_{j}$.
          \item $f^{j+1}(\partial K_{j})\subseteq D_{-1}$ for all $j\geq 0$.
        \end{enumerate}
     \end{prop}
           
     \begin{proof}
  We construct $f$ as the limit of a sequence of rational functions
    $(f_j)_{j=0}^{\infty}$,
   which are defined inductively using Runge's theorem. More precisely, 
   for $j\geq 1$, the function $f_j$ approximates a function $g_j$, defined and meromorphic
   on a neighbourhood of a compact subset $A_j\subseteq\C$, 
   up to an error of at most $\eps_j> 0$. The function $g_j$
   in turn is defined in terms of the previous function $f_{j-1}$.

\begin{figure}
\begin{center}
\def\svgwidth{\textwidth}
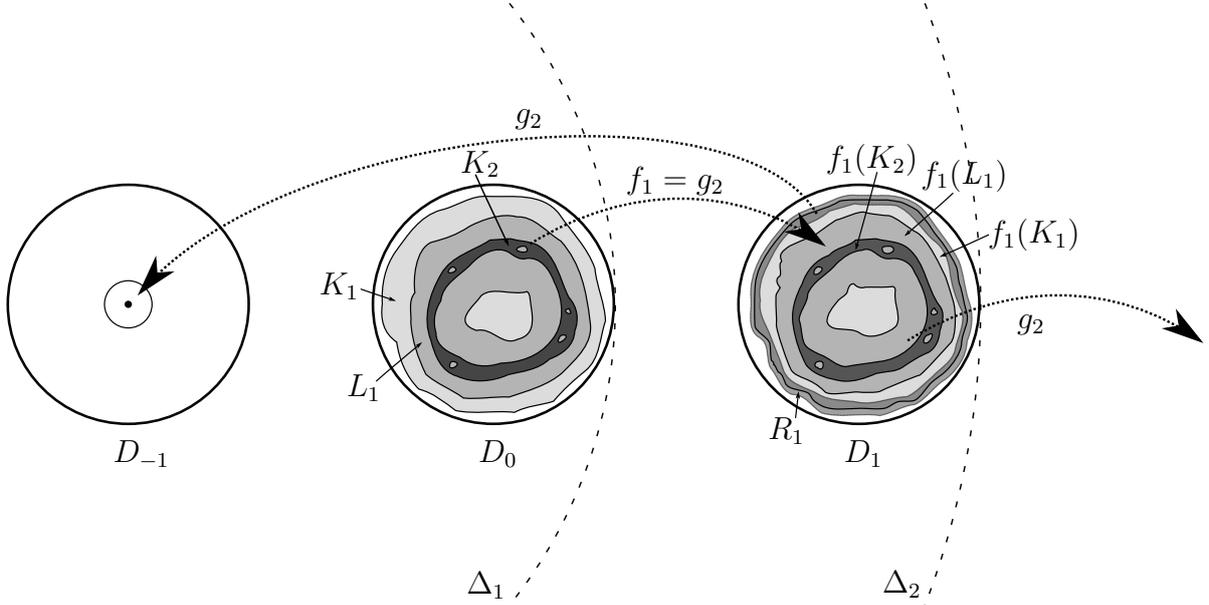
\end{center}
\caption{The construction of $g_{j+1}$ in the proof of Proposition~\ref{prop:main}, for $j=1$.
 The approximation of $g_2$ by a rational function $f_2$ will introduce poles between $f_1(L_1)$ (which is mapped into $D_2$ by $g_2$) and $R_1$ 
  (which is mapped inside $D_{-1}$).}\label{fig:proofsketch}
\end{figure}

    For $j\geq 0$, define 
     \[ \Delta_j       \defeq \overline{D(-3,1+3j)} \supseteq D_{j-1}.\]
     The set $A_j$ will satisfy $A_j\supseteq \Delta_{j-1}$ for $j\geq 1$, so in particular
        \begin{equation}\label{eqn:fillsplane} \bigcup_{j\geq 1} A_j = \bigcup_{j\geq 0} \Delta_j = \C. \end{equation}

    The inductive construction ensures that $f_j$ and $A_j$ have the following properties for $j\geq 0$.
    \begin{enumerate}[(i)]
      \item 
         $f_{j}^j$ is defined, holomorphic and injective on a neighbourhood of $K_j$, and $f_{j}^j(K_j)\subseteq D_j$.\label{item:injectivity}
      \item $f_j^k(K_j)\subseteq \interior(\Delta_j)$ for $0\leq k<j$. \label{item:iterates}
      \item $\eps_1 < 1/2$ and $\eps_{j} \leq \eps_{j-1}/2$ for $j\geq 2$.\label{item:inductiveclose}
    \end{enumerate}
 
  To anchor the induction, we set
     $f_{0}(z) \defeq -3$ for $z\in\C$. Recall that $K_0\subseteq D_0$ by assumption and hence \ref{item:injectivity}  holds for $j=0$. 

   Let $j\geqslant 0$ and suppose that $f_{j}$ has been defined, and that $\eps_{j}$ has been defined if $j\geq 1$. 
     Set $Q_{j} \defeq f_{j}^{j}(\partial K_{j})$. Choose a compact neighbourhood $R_j\subset D_j$ of
        $Q_j$ that is disjoint from $f_j^j(K_{j+1})$; this is possible by~\ref{item:injectivity}. 
        Also choose a compact neighbourhood
         $L_j\subset K_j$ of $K_{j+1}$ such that $f_j^j(L_j)\cap Q_j = \emptyset$. Define the compact set
         \begin{equation}\label{eqn:Aj} A_{j+1} \defeq \Delta_{j} \cup R_{j} \cup f_{j}^{j}(L_{j}). \end{equation}
       Since $D_j \cap \Delta_j = \emptyset$, the three sets in the definition of $A_{j+1}$ are pairwise disjoint. 
      We define
         \[ g_{j+1} \colon A_{j+1} \to \Ch;\quad z\mapsto
             \begin{cases}
                 f_{j}(z), & \text{if } z\in \Delta_{j}, \\
                 -3, &\text{if }z\in R_{j}, \\ 
                 z+3, &\text{if } z\in f^{j}_{j}(L_{j}). \end{cases} \] 
By definition, the function $g_j$ extends meromorphically to a neighbourhood of $A_j$. 
       Observe that $g_{j+1}^j = f_j^j$ on $K_j$ by~\ref{item:iterates}; hence 
        $g_{j+1}^{j+1}$ is defined and injective on $L_{j}$ by~\ref{item:injectivity},
        with $g_{j+1}^{j+1}(L_j) \subset D_{j+1}$, and $g_{j+1}^j(\partial K_j) = f_j^j(\partial K_j)\subset R_j$.
        
         Now choose $\eps_{j+1}>0$ according to~\ref{item:inductiveclose}
       and sufficiently small that  
       any meromorphic function~$f$ that satisfies
         $\lvert f(z) - g_{j+1}(z)\rvert \leq 2\eps_{j+1}$ on $A_{j+1}\setminus g_{j+1}^{-1}(\infty)$ satisfies: 
       \begin{enumerate}[(1)]
          \item $f^{j}(\partial K_{j})\subset R_{j}$.\label{item:pjclose}
          \item $f(R_{j}) \subseteq D_{-1}$.\label{item:Bjclose}
          \item $f^{j+1}$ is defined and injective on a neighbourhood of $K_{j+1}\subseteq \interior(L_{j})$, where it satisfies
             $f^{j+1}(K_{j+1})\subset D_{j+1}$.\label{item:fjinjective}
          \item $\bigcup_{k=0}^j f^{k}(K_{j+1})\subset \interior(\Delta_{j+1})$. 
       \end{enumerate}
       Since $g_{j+1}$ satisfies all of these properties, this is possible by Lemma~\ref{lem:approx}.
       We now let $f_{j+1}\colon \C\to\hat{\C}$ be
       a rational map approximating $g_{j+1}$ up to an error of at most $\eps_{j+1}$, 
       according to Theorem~\ref{thm:Runge}. Then $f_{j+1}$ satisfies the properties \ref{item:injectivity} and \ref{item:iterates}. This completes the inductive 
       construction. 
              
    Note that, for $j\geq k$, $\Delta_k \subset \Delta_j \subset A_{j+1}$ by~\eqref{eqn:Aj}. 
      Condition~\ref{item:inductiveclose} implies that $(f_j)_{j=k+1}^{\infty}$ forms
     a Cauchy sequence on every $A_k$, so by~\eqref{eqn:fillsplane} the functions $f_j$
     converge uniformly in the spherical metric to a meromorphic function $f$. For $1\leq k\leq j$,
     \[
     |f_j(z)-g_k(z)|\leq \eps_k+ \dots + \eps_j \leq 2\eps_k\quad \textup{for all } z\in A_k.
     \] 
     Hence the limit function $f$ satisfies
     $\lvert f(z) - g_k(z)\rvert\leq 2\eps_k$ for all $z\in A_k$ and $k\geq 1$. 

    Since $g_{1}(D_{-1})=f_{0}(D_{-1})=\{-3\}$, and 
       $2\eps_1 < 1$, it follows that
         $f(\overline{D_{-1}})\subseteq D_{-1}$.  
        Moreover, 
    $f^{j+1}(\partial K_{j})\in D_{-1}$ for $j\geq 0$ by~\ref{item:pjclose} and~\ref{item:Bjclose}. Finally, by~\ref{item:fjinjective}, 
    $f^j$ is injective on $K_j$ and $f^j(K_j)\subseteq D_j$.     This completes the proof of Proposition~\ref{prop:main}.
\end{proof}
\begin{rmk}
By the maximum modulus theorem, the approximating function $f_{j+1}$ must have a pole between $f_j^j(L_j)$ and any connected component of $ f_j^j(\partial K_j)$. 
  These poles lie in $D_j\subset \Delta_{j+1}$, and hence will also be poles of $f_k$ for $k>j+1$, and hence of $f$. It is well-known that the approximating function in Runge's theorem (Theorem~\ref{thm:Runge})
  can be chosen to have at most one pole in every complementary component of the approximating set $A$. For the unbounded complementary component of $A_{j+1}$, we may assume this pole
  to be at infinity, so that the limiting function $f$ will not have any other poles.
\end{rmk}

\begin{proof}[Proof of Theorem~\ref{thm:main}] Let $f$ be the meromorphic function from~Proposition~\ref{prop:main}.
   Since $K\subseteq K_j$ for all $j\geq 0$, we have $f^j(K)\subseteq D_{j}$ for all $j\geq 0$. 
      In particular, $f^j\vert_K\to \infty$ uniformly as $j\to\infty$, so $K\subset I(f)$
      and $\interior(K)\subseteq F(f)$ by Montel's theorem. 
      
       On the other hand, we have  $f^k(\partial K_{j-1})\subset D_{-1}\subset F(f)$ for $k\geq j$. In particular, since every point of $\partial K$ is the limit of points of $\partial K_j$, 
         $(f^n)_{n=0}^{\infty}$ is not normal at any point of $\partial K$; so $\partial K \subset J(f)$. Moreover,
  for every $z\in \C\setminus K$, there exists $j\in \mathbb{N}$ such that $K$ and $z$ are in different connected components of $\mathbb{C}\setminus \partial K_j$. Since $\partial K_j\subset \C\setminus (J(f)\cup I(f))$, it follows that every connected component of $K$ is a connected
         component of $I(f)$, and every connected component of $\partial K$ is a connected component of $J(f)$. 
         Furthermore, every connected component of $\interior(K)$ is a Fatou component of $f$, which is a wandering domain
         since $f^n(K)\cap f^m(K)= \emptyset$ for $n\neq m$. 
 \end{proof}

To conclude the paper, we prove Theorems~\ref{thm:comp-fatou} and \ref{thm:comp-julia}. Recall that Theorem~\ref{thm:comp-fatou} states that a bounded domain $U\subseteq \C$ is a Fatou component of a meromorphic function if and only if $U$ is regular, while Theorem~\ref{thm:comp-julia} states that a continuum $X\subseteq \C$ is a Julia component of a meromorphic function if and only if $X$ has empty interior.

\begin{proof}[Proof of Theorem~\ref{thm:comp-fatou}]
By Lemma~\ref{lem:regular-fatou-comp}, a bounded Fatou component $U$ is necessarily regular. Note that since $U$ is bounded, $F(f)\setminus U\neq\emptyset$. In the other direction, if 
 $U$ is a bounded regular domain, then $K=\overline{U}$ is a compact set such that $U$ is a connected component of
  $\textup{int}(K)$. Thus, it follows from Theorem~\ref{thm:main} that there is a transcendental mero\-morphic function $f$ such that $U$ is a wandering domain of~$f$.
\end{proof}

\begin{proof}[Proof of Theorem~\ref{thm:comp-julia}]
 By the blowing-up property of the Julia set, if $J(f)\neq\C$, then every connected component of $J(f)$ has empty interior.
  Conversely, suppose that $X\subset\C$ is a continuum with $\interior(X)=\emptyset$. Then we can apply Theorem~\ref{thm:main} with
   $K=X$; for the resulting function $f$, the set $X = \partial K$ is a component of $J(f)$, as claimed. 
\end{proof}

\section*{Conflict of interest statements}
On behalf of all authors, the corresponding author states that there is no conflict of interest.

\section*{Data availability}
Data sharing is not applicable to this article as no datasets were generated or analysed during the current study.

\bibliographystyle{amsalpha}

\bibliography{bibliography}

\end{document}